\documentclass[draft]{amsart}

\usepackage{mathtools}
\usepackage{amssymb}
\usepackage{hyperref}
\usepackage{imakeidx}
\usepackage{tikz}
\usepackage{manfnt}
\usetikzlibrary{%
  matrix,%
  calc,%
  arrows%
}

\newcommand{\C}{\mathbb{C}}

\newcommand{\N}{\mathbb{N}}

\newcommand{\Q}{\mathbb{Q}}
\newcommand{\R}{\mathbb{R}}

\newcommand{\Z}{\mathbb{Z}}
\newcommand{\comment}[1]{}

\makeatletter
\def\imod#1{\allowbreak\mkern10mu\left({\operator@font mod}\,\,#1\right)}
\makeatother

\newtheorem{thm}{Theorem}[section]
\newtheorem{cor}[thm]{Corollary}
\newtheorem{prop}[thm]{Proposition}
\newtheorem{lem}[thm]{Lemma}

\theoremstyle{remark}
\newtheorem{rem}[thm]{Remark}

\theoremstyle{definition}

\numberwithin{equation}{section}
\numberwithin{thm}{section}

\newtheorem{thm*}{Theorem}
\newtheorem{lem*}{Lemma}
\newtheorem{prop*}[thm]{Proposition}
\newtheorem{conj*}{Conjecture}

\newtheorem{Claim*}{Claim}
\newtheorem{defn}[thm]{Definition}
\newtheorem{defn*}{Definition}
\newtheorem{tble}[thm]{Table}

\begin{document}






\title{Rational singularities and Uniform Symbolic Topologies}


%

\author{Robert M. Walker}

\address{Department of Mathematics, University of Michigan, Ann Arbor, MI, 48109}
\email{robmarsw@umich.edu}


%

\thanks{I wish to thank my thesis adviser, Karen E. Smith, for several hours of fruitful, encouraging conversation. I also thank Daniel Hern\'{a}ndez, Felipe P\'{e}rez, and Luis N\'{u}\~{n}ez-Betancourt for each critiquing a draft of the paper. I thank the anonymous referee for suggestions improving the exposition. I was supported by a NSF GRF under Grant Number PGF-031543, and this work was partially supported by the NSF RTG grant 0943832.}

%

\begin{abstract}
Take $(R, \mathfrak{m})$ any normal Noetherian domain, either local or $\N$-graded over a field. We study the question of when $R$ satisfies the uniform symbolic topology property (USTP) of Huneke, Katz, and Validashti:
 namely, that there exists an integer $D>0$ such that for all prime ideals $P \subseteq R$, the symbolic power $P^{(Da)} \subseteq P^a$ for all $a >0$. Reinterpreting results of Lipman, we deduce that when $R$ is a two-dimensional rational singularity, then it satisfies the USTP. 
Emphasizing the non-regular setting, we produce explicit, effective multipliers $D$, working in two classes of surface singularities in equal characteristic over an algebraically closed field, using: (1) the volume of a parallelogram in $\R^2$ when $R$ is the coordinate ring of a simplicial toric surface; or (2) known invariants of du Val isolated singularities in characteristic zero due to Lipman.
\end{abstract}

\maketitle



\section{Introduction and Conventions for the Paper}

Given a Noetherian commutative ring $R$, an important open problem is discerning when there is an integer $D = D(R) > 0$, depending on $R$, such that the symbolic power $P^{(D a)} \subseteq P^a$ for all prime ideals $P \subseteq R$ and all integers $a > 0$. 
When such a $D$ exists, 
one says that $R$ satisfies the \textbf{uniform symbolic topology property (USTP)} on prime ideals \cite{HKV2}. 
There are several simple cases where $D = 1$ will suffice, that is, where symbolic powers and ordinary powers coincide for all prime ideals. Namely $D=1$ works if $R$ is: (a)  Artinian; or (b) a one-dimensional domain; or (c) a two-dimensional UFD ($\operatorname{Cl}(R) =0$ in Lemma \eqref{thm: du Val bound 1} below). In particular, $D = 1$ works for any regular local ring of dimension at most two. 

 The Ein-Lazarsfeld-Smith Theorem  \cite{ELS, HH1}, as extended by Hochster and Huneke, says that if $R$ is a $d$-dimensional regular ring containing a field, then $P^{(da)} \subseteq P^a$ for all prime ideals $P \subseteq R$ and all integers $a > 0$; 
the papers \cite{HH2,MRJ1,TaYo1}  
extend what is known for regular rings containing a field. 
In stark contrast, effective uniform bounds are harder to unearth in the non-regular setting. 
In this direction, Huneke, Katz, and Validashti \cite[Cor~3.10]{HKV} give a non-constructive proof, under mild hypotheses on the ground field, that a uniform $D$ exists when $R$ is a reduced isolated singularity containing a field. 
As far as the author knows, no constructive (or even sharp) bounds $D$ were known prior to cases covered in this paper.

The following is a simple main case of the key lemma we use to produce effective bounds:

\begin{lem}\label{thm: du Val bound 1}
Let $R$ be a Noetherian normal domain whose global divisor class group $\operatorname{Cl}(R) : = \operatorname{Cl(Spec}(R))$ is finite abelian of order $D$. Then for all ideals $\mathfrak{q} \subseteq R$ of pure height one, the symbolic power $\mathfrak{q}^{(Da)} \subseteq \mathfrak{q}^a$ for all $a> 0$. 
\end{lem}

\noindent 
Lemma \ref{thm: du Val bound 1} supports Huneke's philosophy in \cite{hun1} that there are uniform bounds lurking in Noetherian rings. In particular, when $R$ in Lemma \ref{thm: du Val bound 1} has Krull dimension two, it satisfies the uniform symbolic topology property on prime ideals, since height one primes are the only nontrivial class to check.
Notice that the $R$ in Lemma \ref{thm: du Val bound 1} is allowed to be of \textbf{mixed characteristic}, unlike all previous results which require the ring to contain a field. 

It remains to identify classes of rings that satisfy the hypotheses of Lemma \ref{thm: du Val bound 1}. In dimension two, we use the following result due to Lipman:
\begin{thm*}[Lipman {\cite[Prop~17.1]{Lip0}}]\label{thm: rational singularity 1}
\textit{Let $(R, \mathfrak{m})$ be a  normal Noetherian local domain of dimension two.  If $R$ has a rational singularity, then $\operatorname{Cl}(R)$ is finite.} 
\end{thm*}
\noindent Lipman defines a two-dimensional, normal Noetherian local domain $(R,  \mathfrak{m})$ to have \textbf{rational singularities} if there is a proper, birational map $f \colon X \to \operatorname{Spec}(R)$ from a regular scheme $X$ such that $H^1 (X, \mathcal{O}_X) = 0$. 
 This allows us to affirmatively answer a question of Huneke, Katz, and Validashti \cite[Question~1.1]{HKV2} for rationally singular surfaces: 
\begin{thm}\label{cor: rational singularity 2}
All two-dimensional rational singularities $(R, \mathfrak{m})$ satisfy the uniform symbolic topology property on primes. 
\end{thm}

\noindent In what other cases can we find effective values of $D$? In subsection \ref{subsection: rational singularities}, we will use results of Lipman to identify explicit effective values $D$ for du Val singularities over $\C$; see Table \ref{tble:duVal01}. Interestingly, our sharp results show that the smallest $D$ that works for a class of non-regular rings can be arbitrarily large, even if the Krull dimension and multiplicity of the singularity are fixed. 

To produce effective bounds in higher dimension, we turn to the class of normal toric rings containing an arbitrary algebraically closed field $k$. These rings are built from a convex-geometric starter: precisely, a polyhedral cone $\sigma \subseteq \R^n$ generated by a finite set of \textbf{primitive} lattice points $v_i \in \Z^n$, each having coordinates \textbf{whose gcd is one}. It turns out that a toric ring has finite class group
precisely when it is simplicial, that is, the $v_i$ are $\R$-linearly independent \cite[Prop~4.2.2+Prop~4.2.7]{torictome}. In particular, we have a 
\begin{thm}\label{thm: det bound 1}
Let $\sigma \subseteq \R^n$ be a full-dimensional simplicial strongly convex rational polyhedral cone, and let $R = k[\sigma^\vee \cap \Z^n]$ be the corresponding toric ring over an algebraically closed field $k$. Then $R$ satisfies the uniform symbolic topology property on ideals of pure height one with multiplier $D$, where $D$ is the volume of the $n$-parallelotope spanned by the primitive generators of $\sigma$.  
\end{thm}
\noindent In particular, this volume can be computed as the determinant of a positive-definite integer matrix. See subsection \ref{subsection: Proof 1} for details. 


\noindent \textbf{Conventions:} All our rings are commutative with identity. From subsection \ref{section: MainTheorem} onwards, and except when stated otherwise, all our rings $R$ will be Noetherian normal domains, and $k$ will denote an algebraically closed field. All our algebraic varieties over $k$ are irreducible.  When we say $R$ is \textbf{graded}, we mean that 
$R = \bigoplus_{d \ge 0} R_d$ is graded by $\N$, with $R_0$ being a field, and $\mathfrak{m} = \bigoplus_{d > 0} R_d$ the unique homogeneous maximal ideal.  


\section{Local Uniform Annihilation and Symbolic Powers}\label{section: preliminaries 1}

\noindent \textbf{Defining Symbolic Powers.} If $P$ is a prime ideal in a nonzero Noetherian ring $R$, its \textbf{$a$-th ($a \in \Z_{>0}$) symbolic power}
$$P^{(a)} := P^a R_P \cap R$$ is the unique $P$-primary component in any minimal primary decomposition of $P^a$. 
More generally, if $I$ is any proper ideal of $R$, its \textit{$a$-th symbolic power} is  
$$I^{(a)} := I^a W^{-1} R \cap R, \mbox{ where }W = R - \bigcup \{P \colon P \in \operatorname{Ass}_R (R/I)\}.$$ 
While $I^a \subseteq I^{(a)}$ for all $a$, the converse can fail for $a>1$: $I^{(1)} = I$ since $W$ is the set of nonzerodivisors modulo $I$.

\subsection{Uniform Annihilation and Class Groups}\label{section: MainTheorem} 

Given a normal Noetherian domain $R$, the \textbf{divisor class group} $\operatorname{Cl}(R) = \operatorname{Cl}(\operatorname{Spec}(R))$ of $R$ is the free abelian group on the set of height one prime ideals of $R$ modulo relations $$a_1 P_1 + \ldots + a_r P_r = 0,$$ whenever the ideal $P_1^{(a_1)} \cap \ldots \cap 
P_r^{(a_r)}$ is principal (see Hochster's lecture notes \cite{hoch2} for more details on this definition; see also Hartshorne's presentation \cite[Ch.II,$\S$6]{Hartsh0}). 
It is well known that the class group $\operatorname{Cl}(R)$ is trivial if and only if $R$ is a UFD \cite[Ch.II,$\S$6]{Hartsh0}. Both conditions mean that every height one prime ideal in $R$ is principal. We now show this latter assertion is equivalent to saying that all symbolic powers of a height one prime ideal $P \subseteq R$ are principal, so $P^{(a)} = P^a$ for all $a>0$, and all height one primes $P$ in a UFD. We follow an anonymous referee's advice to state a 
\begin{lem}
If $R$ is an arbitrary Noetherian ring, and $P = (f)$ is prime with $f$ a nonzerodivisor, then $P^{(a)} = (f^a)$ for all $a>0$.  
\end{lem}
\begin{proof}
Induce on $a$ with base case $a=1$. Assuming the statement for some $a \ge 1$, take $x \in P^{(a+1)}$. Since $x  \in P^{(a+1)} \subseteq P^{(a)} = (f^a)$, $x = f^a y$ for some $y$. By the choice of $x$, there is $s \not\in P$ with $s x = s f^a y \in P^{a+1} = (f^{a+1})$. Since $f$ is a nonzerodivisor,  $sy \in (f)  = P$, which is prime. Therefore, $y \in P = (f)$, and $x \in (f^{a+1})$. 
\end{proof}

Given a Noetherian reduced ring $R$, recall that the affine scheme $X = \operatorname{Spec}(R)$ is \textbf{locally factorial} if $R_\mathfrak{p}$ is a UFD for all prime ideals $\mathfrak{p} \subseteq R$; or equivalently, if $\operatorname{Cl}(R_\mathfrak{p}) = 0$ for all primes $\mathfrak{p} \subseteq R$. 

\begin{defn}\label{def: uniformly torsion domain}
A reduced Noetherian normal ring $R$ is \textbf{(locally) uniformly annihilated} if there exists an integer (multiplier) $D > 0$ such that one, and hence all, of the following equivalent conditions will hold: 
\begin{enumerate}
\item $D \cdot \operatorname{Cl}(R_\mathfrak{p}) = 0$  for all prime ideals $\mathfrak{p} \subseteq R$. More precisely, $P^{(D)}R_\mathfrak{p} = (PR_\mathfrak{p})^{(D)}$ is principal for all height one primes $P \subseteq \mathfrak{p}$. 
\item The annihilator ideal $\operatorname{Ann}_\Z (\operatorname{Cl}(R_{\mathfrak{p}})) \supseteq D \Z$ for all prime ideals $\mathfrak{p} \subseteq R$.
\item $D \cdot \operatorname{Cl}(R_\mathfrak{m}) = 0$  for all maximal ideals $\mathfrak{m} \subseteq R$. 
\end{enumerate}
\end{defn}

\noindent Notice that (3) implies (1) since principal ideals remain principal when we extend along the ring map $R_{\mathfrak{m}} \to R_{\mathfrak{p}}$.   
Note that since one can take $D=1$ at locally factorial points of $\operatorname{Spec}(R)$, it suffices to compute $D$ at maximal ideals $\mathfrak{m} \subseteq R$ such that $R_\mathfrak{m}$ is not a UFD. In language more familiar to algebraic geometers, $D$ annihilates the local class group $\operatorname{Cl}_{loc}(X) = \operatorname{Cl}(X) / \operatorname{Pic}(X)$ of $X = \operatorname{Spec}(R)$.  

\begin{rem}
A Noetherian normal domain $R$ is uniformly annihilated when the annihilator ideal $\operatorname{Ann}_\Z (\operatorname{Cl}(R)) \neq 0$ (e.g., if $\operatorname{Cl}(R)$ is finite). However, the smallest local uniform multiplier need not annihilate  $\operatorname{Cl}(R)$ globally. For example, if the Dedekind domain $R = \Z [\sqrt{-5}]$, then $\operatorname{Cl}(R) \cong \Z / 2 \Z$, while any Dedekind domain $R$ is locally factorial, so $\operatorname{Cl}(R)/\operatorname{Pic}(R)$ is trivial in the language of algebraic geometry, and $D=1$ is the optimal multiplier.
\end{rem}

\begin{rem}\label{rem:localgradedcyclic} If $R$ is local, then the smallest uniform multiplier would generate $\operatorname{Ann}_\Z (\operatorname{Cl}(R))$. When $R$ is  $\N$-graded over a field with unique graded maximal ideal $\mathfrak{m}$, $\operatorname{Cl}(R) \cong \operatorname{Cl}(R_{\mathfrak{m}})$ \cite[Lem~5.1]{Murthy1} and we again consider a generator of  $\operatorname{Ann}_\Z (\operatorname{Cl}(R))$. In each case, this optimal multiplier $D$ is the class group's order if and only if the group is finite cyclic; we fill in Table \ref{tble:duVal01}'s final column via this fact.
\end{rem} 

\subsection{The Main Lemma}\label{subsection: main lemma}  Symbolic powers are notoriously difficult to compute by hand, i.e., to give generators for. Working in height one, item (a) of the following proposition allows us to study them indirectly. Note that if $I$ is a proper ideal in a Noetherian ring, we say that $I$ has \textbf{pure height} 
$h$ if all of its associated primes have height $h$, in particular, none are embedded.

\begin{prop}\label{prop: unmixed ideals 1}
Let $R$ be a Noetherian normal domain, and $\mathfrak{q}$ any ideal of pure height one with associated primes $P_1, \ldots, P_r$. Then: 
\begin{enumerate}
\item[(a)] There exist positive integers $b_1, \ldots, b_r$, uniquely determined by $\mathfrak{q}$, such that the symbolic power $\mathfrak{q}^{(E)} = P_1^{(E b_1)} \cap \ldots \cap P_r^{(E b_r)}$ for all $E > 0$. 
\item[(b)] If either $(1)$ $D \cdot \operatorname{Cl}(R) = 0$, or simply $(2)$ $\mathfrak{q}^{(D)}$ is principal,  
then for all integers $a>0$, $\mathfrak{q}^{(Da)} = (\mathfrak{q}^{(D)})^a$ is principal and  $\mathfrak{q}^{(Da)} \subseteq \mathfrak{q}^a$.
\end{enumerate}
\end{prop}
\begin{proof} First, we prove (a). 
For each $i$, the local ring $S_i = R_{P_i}$ is a discrete valuation ring, and we let $t_i\in S_i$ be a local uniformizing parameter. Then $S_i$ is a PID, so an ideal $J \subseteq S_i$ is $P_i S_i = (t_i) S_i$-primary if and only if $J = (P_i S_i)^n = P_i^n S_i = (t_i^n)S_i$ for some $n>0$. In particular, $\mathfrak{q} S_i$ is $P_i S_i$-primary, say $\mathfrak{q} S_i = (t_i^{b_i}) S_i$. Then the $P_i$-primary component of $\mathfrak{q}$ is $P_i^{(b_i)}$.  Thus 
$\mathfrak{q} = P_1^{(b_1)} \cap \ldots \cap P_r^{( b_r)}$ and clearly the $b_i$ are uniquely determined by $\mathfrak{q}$. Similarly, $\mathfrak{q}^E S_i = (t_i^{E b_i}) S_i$ for $E > 0$, so $\mathfrak{q}^{(E)} = P_1^{(E b_1)} \cap \ldots \cap P_r^{(E b_r)}$ for all $E > 0$. 

For (b), first note that (1) implies (2): indeed, since $\mathfrak{q} = P_1^{(b_1)} \cap \cdots \cap P_r^{(b_r)}$, it yields an element $[\mathfrak{q}] := b_1 [P_1] + \cdots +  b_r [P_r] \in \operatorname{Cl}(R)$, and since $0 = D [\mathfrak{q} ] = [\mathfrak{q}^{(D)}] \in \operatorname{Cl}(R)$, we conclude that $\mathfrak{q}^{(D)}$ is principal. So we proceed assuming $\mathfrak{q}^{(D)}$ is principal. Since $\mathfrak{q}^{(D)} \subseteq \mathfrak{q}^{(1)} = \mathfrak{q}$, by taking $a$-th powers, part (b) follows in full once we explain how $\mathfrak{q}^{(Da)} = (\mathfrak{q}^{(D)})^a$ for all  integers $a>0$.  Indeed, using the notation in the proof of (a),  $\mathfrak{q}^{(Da)}S_i = (\mathfrak{q}^{(D)})^a S_i = (t^{D  a b_i}) S_i$ for all $i$, and we simply contract back to $R$.   
\end{proof}
\noindent 

We now prove the main lemma of the paper, expanding on Lemma \ref{thm: du Val bound 1}:

\begin{lem}\label{thm: du Val bound 2}
Let $R$ be a reduced Noetherian normal ring. Suppose $D$ annihilates $\operatorname{Cl}(R_\mathfrak{m})$ for all maximal ideals $\mathfrak{m}$ in $R$. Then for all $a> 0$ and all ideals $\mathfrak{q} \subseteq R$ of pure height one, the symbolic power $\mathfrak{q}^{(D a)} \subseteq \mathfrak{q}^a$. 
\end{lem}


\begin{proof} 
First, we reduce to the local case. Indeed, recall that given two ideals $I , J$ in $R$, the inclusion $I \subseteq J$ holds (that is, the $R$-module $\frac{J + I}{J} = 0$) if and only if $I R_\mathfrak{m} \subseteq J R_\mathfrak{m}$ (that is, the $R_\mathfrak{m}$-module $\frac{(J + I) R_\mathfrak{m}}{J R_\mathfrak{m}} = \frac{J R_\mathfrak{m} + I R_\mathfrak{m}}{J R_\mathfrak{m}}= 0$) for all maximal ideals $\mathfrak{m} \subseteq R$. So we may assume $R$ is a normal Noetherian local domain (in keeping with $R$ being uniformly annihilated). 
Then, for $R$, $D$, and all $\mathfrak{q}$ as stated, Proposition \ref{prop: unmixed ideals 1}(b) gives us the inclusions. 
\end{proof}

\section{Effective Uniform Bounds}\label{section: DetBound} 

We begin with a selective review of toric geometry. For more details (e.g., omitted definitions, proofs, and relevant exercises) we refer the reader to: Chapters 1,3,4 of Cox, Little, and Schenck \cite{torictome}; or alternatively, Chapters 1 and 3 of Fulton  \cite{introtoric}.\footnote{While both references work over $\C$, the facts we review in this paper hold over any algebraically closed field.} 
For simplicity, we work with the standard lattice $\Z^n$ in $\R^n$, and a fixed algebraically closed field $k$.

Any normal affine toric $n$-fold can be obtained from a strongly convex, rational polyhedral cone $\sigma$ in $\R^n$, that is, a closed, convex set 
$$\sigma = \operatorname{Cone}(G) = \left\lbrace\sum_{v \in G} a_v \cdot v \colon \mbox{ each }a_v \in \R_{\ge 0}\right\rbrace  \subseteq \R^n,$$ 
containing no line through the origin (strong convexity) and generated by a (possibly empty) finite set $G$  of nonzero vectors in $\Z^n$ (rationality). 
The cone has \textbf{dimension} $\dim \sigma := \dim (\R\mbox{-linear span of }G) \le n;$ if $\dim \sigma  =  n$, then $\sigma$ is called \textbf{full}.
For such a cone $\sigma = \operatorname{Cone}(G) \subseteq \R^n$, the \textbf{dual} $$\sigma^\vee := \{w \in \R^n \colon \langle w , v \rangle \ge 0 \mbox{ for all }v \in G\} \quad \mbox{($\langle \cdot , \cdot \rangle$ is dot product in $\R^n$)}$$ is rational polyhedral and full, from which we obtain
\begin{enumerate}
\item A finitely-generated semigroup under addition $$(S_\sigma, +) : = \sigma^\vee \cap \Z^n = 
\left\lbrace \sum_{i=1}^r a_i m_i \colon \mbox{ each }a_i \in \Z_{\ge 0}\right\rbrace $$ for some finite list of generators $m_1, \ldots, m_r \in \sigma^\vee \cap \Z^n$. 
\item A normal domain of finite type over $k$: namely, the semigroup ring $R = k[S_\sigma]$ generated as a $k$-vector space by the characters $\chi^m$ with $m \in S_\sigma$. Note that $k[S_\sigma]$ is generated as a $k$-algebra by characters $\chi^{m_1}, \ldots, \chi^{m_r}$ where the $m_i$ are the semigroup generators of $S_\sigma$ in (1). 
Any domain so  obtained is a \textbf{(normal) toric ring} (or normal affine semigroup algebra, by definition) over $k$.
\item A normal toric $n$-fold, $U_\sigma = \operatorname{Spec}(k[S_\sigma])$. 
\end{enumerate}
The cone $\sigma \subseteq \R^n$ as above is \textbf{simplicial} if $\sigma = \{0\}$ or $\sigma = \operatorname{Cone}(G)$ for some $G \subseteq \Z^n$  forming part of a $\R$-basis for $\R^n$, and we call the corresponding toric ring and toric variety \textbf{simplicial}. 
When $\sigma \neq \{0\}$, we may assume $G = \{v_1, \ldots, v_{\dim(\sigma)}\}$ consists of \textbf{primitive} vectors, i.e., that the coordinates of each $v_i$ have no common prime integer factor. When $\sigma$ is full, we index the $v_i$ so the matrix $A_G \in \mbox{Mat}_{n\times n} (\Z)$ whose $i$-th row is $v_i$ has determinant $d  = \det(A_G) \in \Z_{>0}$, the volume of the $n$-parallelotope spanned by $v_1, \ldots, v_n$. 

Working over an algebraically closed field $k$, each ray $\rho \in \Sigma(1)$, the collection of \textbf{rays} (one-dimensional faces) of $\sigma$,  has a generator $u_\rho \in \rho \cap \Z^n$ that is primitive, and yields a torus-invariant prime divisor $D_\rho$ on $X = U_\sigma$, per the Orbit-Cone Correspondence \cite[Thm~3.2.6]{torictome}.
The torus-invariant Weil divisors on $X$ form a free abelian group  \cite[Exercise~4.1.1]{torictome} 
$$\operatorname{Div}_{T} (X) = \bigoplus_{\rho \in \Sigma(1)} \Z D_\rho.$$ 
We close by recording without proof a few facts from sections 4.1 and 4.2 of \cite{torictome} about the divisor class group $\operatorname{Cl}(X)$, specialized to the affine  case.
\begin{thm}\label{thm: exact sequence}
Take a normal affine toric variety $X = U_\sigma$. Then 
\begin{enumerate}
\item $\operatorname{Cl}(X)$ is finite abelian if and only if $\sigma$ is simplicial $($cf.,  \cite[Prop~4.2.2 and Prop~4.2.7]{torictome}$)$. If so, then all Weil divisors on $X$ are $\Q$-Cartier of index bounded by the order of $\operatorname{Cl}(X)$.   
\item If $G = \{u_\rho \colon \rho \in \Sigma(1)\}$ spans $\R^n$, i.e., $\sigma$ is full, then we have a short exact sequence of abelian groups 
\begin{equation}\label{eqn:SES001}
0 \to \Z^n \stackrel{\phi}{\to} \operatorname{Div}_{T} (X) 
\to \operatorname{Cl}(X) \to 0,
\end{equation} 
where $\phi (m) = \operatorname{div}(\chi^m)=  \sum_{\rho \in \Sigma(1)}  \langle m , u_\rho\rangle  D_\rho$, and $\langle \cdot, \cdot \rangle$ is dot product.
\end{enumerate}
\end{thm}


\subsection{The Toric Case}\label{subsection: Proof 1}  The short exact sequence \ref{eqn:SES001} makes it easy to compute the divisor class group $\operatorname{Cl}(U_\sigma)$. 
If $\Sigma(1) = \{\rho_1, \ldots, \rho_r\}$ has $r$ elements,\footnote{Note that $r \ge n$ since $G$ spans $\R^n$ by hypothesis in Theorem \ref{thm: exact sequence}(2).} the map $\phi$ can be treated, up to isomorphism, as a map $\Z^n \to \Z^r$ given by the matrix $A = (u_{\rho_1}, \ldots , u_{\rho_r})^T$ where $T$ denotes transpose: that is, the $i$-th row of $A$ is given by the coordinates of $u_{\rho_i}$. $\operatorname{Cl}(U_\sigma)$ is the cokernel of $\phi$, and hence can be computed up to isomorphism by first finding the Smith normal form of $A$. 
Note that since the alternating sum of the ranks in the short exact sequence vanishes, $\operatorname{Cl}(U_\sigma)$ has rank  
$r - n$. 
We now re-express Theorem \ref{thm: det bound 1}:
\begin{cor}
Let $\sigma = \operatorname{Cone}(G) \subseteq \R^n$ be a simplicial full strongly convex rational polyhedral cone. Let $U_\sigma$ be the corresponding affine toric variety.
Let $A_G \in \operatorname{Mat}_{n \times n} (\Z)$ be the accompanying matrix of primitive ray generators. Then $\operatorname{Cl}(U_\sigma)$ is a finite abelian group of order $d = | \det(A_G) |$.  
\end{cor}

\begin{proof} We maintain the notational conventions of the statement of Theorem \ref{thm: exact sequence}, along with those in the paragraph on simplicial toric varieties in our review above. By hypothesis, 
$G = \{u_\rho \colon \rho \in \Sigma(1)\}$ forms an $\R$-vector space basis of $\R^n$ since $\sigma$ is simplicial. So by Theorem \ref{thm: exact sequence}(2), $\operatorname{Cl}(U_\sigma)$ has rank zero, and hence is finite abelian. Moreover, as a special case of following the Smith normal form approach in the paragraph above, we see that
the matrix $A_G$ defines the action of $\phi$. Thus we conclude that $d$ is the order of $\operatorname{Cl}(U_\sigma)$. 
\end{proof}

\subsection{Rational double points}\label{subsection: rational singularities} 
To obtain additional explicit, effective multipliers, we turn to the case of complete, normal  Noetherian local domains $S$ in equal characteristic zero with du Val (ADE) isolated singularity and algebraically closed residue field; for simplicity, we work with $\C$. 
 In \cite[$\S$24]{Lip0}, Lipman computes the class group isomorphism type (as a $\Z$-module) of each du Val singularity. The du Val (ADE) singularities, also known as \textit{rational double points}, are the most basic isolated surface singularities. 
For a brief introduction to these singularities and their basic properties, we refer the reader to sections III.3 and III.7 of \cite{BHPV00}. Their minimal resolutions can be understood and classified by the simply-laced Dynkin diagrams of types A, D, and E. We can express $S$ as above as the quotient of the power series $\C[[x, y, z]]$ by a single local equation. We now situate a succinct data table, where the last column's entries are the optimal uniform multipliers from Remark \ref{rem:localgradedcyclic}:

 \begin{tble}\label{tble:duVal01}
 $$$$
 \resizebox{\columnwidth}{!}{%
 \begin{tabular}{ l || c | c | c }
     $\stackrel{\mbox{du Val Singularity}}{\mbox{type}}$ & Local Equation & $\stackrel{\mbox{Class group}}{\mbox{(isomorphism type)}}$ & $D_{\min} (S)$ \\ \hline
  $A_n$ ($n \ge 1$)  & $xz  - y^{n+1}$ & $\Z/(n+1)\Z$ & $n+1$\\ \hline
  $D_n$ ($n \ge 4$)  & $x^2 + y z^2 - z^{n-1}$ & $\begin{cases} (\Z/ 2 \Z)^2  & \mbox{$n$ even}\\  \Z/4 \Z & \mbox{$n$ odd}.
\end{cases}$ & $\begin{cases} 2  & \mbox{$n$ even}\\  4 & \mbox{$n$ odd}.
\end{cases}$
  \\ \hline 
  $E_n$ ($n = 6, 7, 8$)  & $\begin{cases} x^4 + y^3 + z^2 & \mbox{ if $n =6$}\\
  x^3y + y^3 + z^2 & \mbox{ if $n =7$}\\ 
  x^5 + y^3 + z^2 & \mbox{ if $n =8$}\\
  \end{cases}$ 
   & $\begin{cases} \Z/3 \Z & \mbox{ if $n =6$}\\
  \Z/2 \Z& \mbox{ if $n =7$}\\ 
  0 & \mbox{ if $n =8$}\\
  \end{cases}$ & $\begin{cases} 3 & \mbox{ if $n =6$}\\
  2 & \mbox{ if $n =7$}\\ 
  1 & \mbox{ if $n =8$}\\
  \end{cases}$
  \\ \hline 
  \end{tabular}%
  }
  \end{tble}
  Note that an analogous data table can be drafted for affine du Val singularity hypersurfaces in $\C^3$ (affine, $\N$-graded case), which arise as the quotients $\C^2/G$ by the action of a finite subgroup $G  \subseteq \mbox{SU}_2 (\C)$. 

Taking the above data table in tandem with Theorem \ref{thm: det bound 1}, we now have access to an infinite supply of concrete, effective uniform bounds that hold in \textbf{non-regular} Noetherian domains whose singularities are sufficiently nice (e.g., toric, du Val). In particular, the Huneke-Katz-Validashti result \cite{HKV} on uniform bounds now has some added company in the setting of rational surface singularities, courtesy of Lipman's work. 

\begin{rem}
When a Noetherian ring $R$ satisfies the uniform symbolic topology property (USTP) on prime ideals, experts might initially expect that the optimal multiplier $D = D_{\min}(R)$ should depend only on simple numerical invariants of $R$, such as Krull dimension, 
or the multiplicity of $R$ at an isolated singularity. However, $A_n$-singularities and $E_8$-singularities have multiplicity two, being rational double points. At one extreme, 
$D_{\min} (A_n) = n+1$ is optimal, grows arbitrarily large with $n$, and does \textbf{not} depend on any such numerical invariants of $A_n$. At the other extreme, $D_{\min} (E_8) = 1$ is lower than both the Krull dimension and the multiplicity. Therefore, a (sharp) uniform bound depending only on such numerical invariants need not exist.
\end{rem}

\section{Lingering Questions and Future Directions}\label{section: Finale} 

To summarize, we have deduced a group-theoretic criterion (Lemma \ref{thm: du Val bound 2}) for uniform symbolic  topologies on ideals of pure height one in a Noetherian normal domain. We also demonstrated the criterion's utility relative to familiar classes of local- or graded Cohen-Macaulay domains with rational singularities 
\cite{hoch0,hoch1}. We close with two natural lines for further investigation. 

\begin{enumerate}
\item Can Lemma  \ref{thm: du Val bound 2} be strengthened to cover all (non-prime) ideals of height one? 

\item Can we identify a candidate mechanism (e.g., group-theoretic) to verify the uniform symbolic topology property for: prime ideals of height one in non-simplicial toric rings; or prime ideals of height two or more, even in the case of simplicial toric rings 
of dimension at least three? 
\end{enumerate}


%

\end{document}